\documentclass{article}

\usepackage[nonatbib,preprint]{neurips_2021}

\usepackage[utf8]{inputenc} % allow utf-8 input
\usepackage[T1]{fontenc}    % use 8-bit T1 fonts
\usepackage{hyperref}       % hyperlinks
\usepackage{url}            % simple URL typesetting
\usepackage{booktabs}       % professional-quality tables
\usepackage{amsfonts}       % blackboard math symbols
\usepackage{nicefrac}       % compact symbols for 1/2, etc.
\usepackage{microtype}      % microtypography
\usepackage{xcolor}         % colors
\usepackage{amsmath}
\usepackage{amsthm}
\usepackage{array}
\usepackage{graphicx}
\usepackage{multirow}
\usepackage{subcaption}
\usepackage{caption} 
\usepackage{wrapfig}
\usepackage{hyperref}
\usepackage{colortbl}
\usepackage[numbers]{natbib}
\title{On Accelerating Distributed Convex Optimizations}

\author{%
  Kushal Chakrabarti \\
  Department of Electrical and Computer Engineering \\
  University of Maryland \\
  College Park, Maryland 20742, U.S.A. \\
  \texttt{kchak@terpmail.umd.edu} \\
  \And
  Nirupam Gupta \\
  \'Ecole polytechnique f\'ed\'erale de Lausanne (EPFL) \\
  CH-1015 Lausanne, Switzerland \\
  \texttt{nirupam.gupta@epfl.ch} \\
  \And
  Nikhil Chopra \\
  Department of Mechanical Engineering \\
  University of Maryland \\
  College Park, Maryland 20742, U.S.A. \\
  \texttt{nchopra@umd.edu} \\
}

\def\R{\mathbb{R}}
\def\D{\mathcal{D}}
\providecommand{\norm}[1]{\ensuremath{\left\lVert#1\right\rVert}}
\providecommand{\mnorm}[1]{\ensuremath{\left\lvert#1\right\rvert}}
\newtheorem{theorem}{\bfseries Theorem}
\newtheorem{lemma}{\bfseries Lemma}
\newtheorem{assumption}{\bfseries Assumption}
\newtheorem*{claim}{\bfseries Claim}

\newcolumntype{C}[1]{>{\centering\let\newline\\\arraybackslash\hspace{0pt}}m{#1}}
\DeclareRobustCommand{\bigO}{%
  \text{\usefont{OMS}{cmsy}{m}{n}O}%
}
\definecolor{Gray}{gray}{0.85}

\begin{document}

\maketitle

\begin{abstract}
This paper studies a distributed multi-agent convex optimization problem. The system comprises multiple agents in this problem, each with a set of local data points and an associated local cost function. The agents are connected to a server, and there is no inter-agent communication. The agents' goal is to learn a parameter vector that optimizes the aggregate of their local costs without revealing their local data points. In principle, the agents can solve this problem by collaborating with the server using the traditional distributed gradient-descent method. However, when the aggregate cost is {\em ill-conditioned}, the gradient-descent method (i) requires a large number of iterations to converge, and (ii) is highly unstable against process noise. We propose an {\em iterative pre-conditioning} technique to mitigate the deleterious effects of the cost function's conditioning on the convergence rate of distributed gradient-descent. Unlike the conventional pre-conditioning techniques, the pre-conditioner matrix in our proposed technique updates iteratively to facilitate implementation on the distributed network. In the distributed setting, we provably show that the proposed algorithm converges {\em linearly} with an improved rate of convergence than the traditional and adaptive gradient-descent methods. Additionally, for the special case when the minimizer of the aggregate cost is unique, our algorithm converges {\em superlinearly}. We demonstrate our algorithm's superior performance compared to prominent distributed algorithms for solving real logistic regression problems and emulating neural network training via a noisy quadratic model, thereby signifying the proposed algorithm's efficiency for distributively solving non-convex optimization. Moreover, we empirically show that the proposed algorithm results in faster training without compromising the generalization performance.
\end{abstract}
\section{Introduction}
\label{sec:intro}

In this paper, we consider solving multi-agent distributed convex optimization problems. Precisely, we consider $m$ agents in the system. The agents can only interact with a common server and the overall system is assumed to be synchronous. Each agent $i \in \{1,\ldots,m\}$ has a set of local data points and a {\em local cost function} $f^i: \R^d \to \R$ associated with its local data points. The aim of the agents is to compute a minimum point of the {\em aggregate cost function} $f:\R^d \to \R$, taking values $f(x) = \sum_{i = 1}^m f^i(x)$ for each $x \in \R^d$, across all the agents in the system. Formally, the goal of the agents is to {\em distributively} compute a common parameter vector $x^* \in \R^d$ such that
\begin{align}
    x^* \in X^* = \arg \min_{x \in \R^d} \sum_{i = 1}^m f^i(x). \label{eqn:opt_1}
\end{align}
Since each agent only partially knows the collective data points, they collaborate with the server for solving the distributed problem~\eqref{eqn:opt_1}. However, the agents \underline{cannot} share their local data points to the server. An algorithm that enables the agents to jointly solve the above problem in the aforementioned settings is defined as a {\em distributed algorithm}.

In principle, the optimization problem~\eqref{eqn:opt_1} can be solved using the distributed gradient-descent (GD) algorithm~\cite{bertsekas1989parallel}. Built upon the prototypical gradient-descent method, several {\em adaptive gradient algorithms} have been proposed, which distributively solve~\eqref{eqn:opt_1}~\cite{nesterov27method, polyak1964some, kingma2014adam, kelley1999iterative,duchi2011adaptive,zeiler2012adadelta,dozat2016incorporating}. Amongst them, some notable distributed algorithms are Nesterov's accelerated gradient-descent (NAG)~\cite{nesterov27method}, heavy-ball method (HBM)~\cite{polyak1964some}, adaptive momentum estimation (Adam)~\cite{kingma2014adam}, BFGS method~\cite{kelley1999iterative}. All these methods are iterative in nature, wherein the server maintains an estimate of a solution defined in~\eqref{eqn:opt_1} and iteratively refines it using the sum of {\em local gradients} computed by the agents.

If the aggregate cost function $f$ is convex, the distributed GD method converges at a rate of $\bigO(1/t)$, where $t$ is the number of iterations~\cite{fessler2008image}. The momentum-based methods, such as NAG, HBM, and Adam, improve upon the convergence rate of GD. In particular, the Adam method has been demonstrated to compare favorably with other optimization algorithms for a wide range of machine learning problems~\cite{wu2016google,radford2015unsupervised,peters2018deep}. However, for general convex cost functions $f$, these algorithms converges at a {\em sublinear} rate~\cite{tong2019calibrating,su2014differential}. For the special case of strong convex cost function $f$, the aforementioned methods converge {\em linearly}~\cite{fessler2008image,tong2019calibrating,su2014differential}. Quasi-Newton methods, such as BFGS, explore second-order information of the cost functions. If the aggregate cost function's Hessian is positive definite at a solution of~\eqref{eqn:opt_1}, then the BFGS method locally converges to a solution at {\em superlinear} rate~\cite{kelley1999iterative}. 

\begin{wrapfigure}{r}{0.5\textwidth}
  \centering
  \includegraphics[width = 0.5\textwidth]{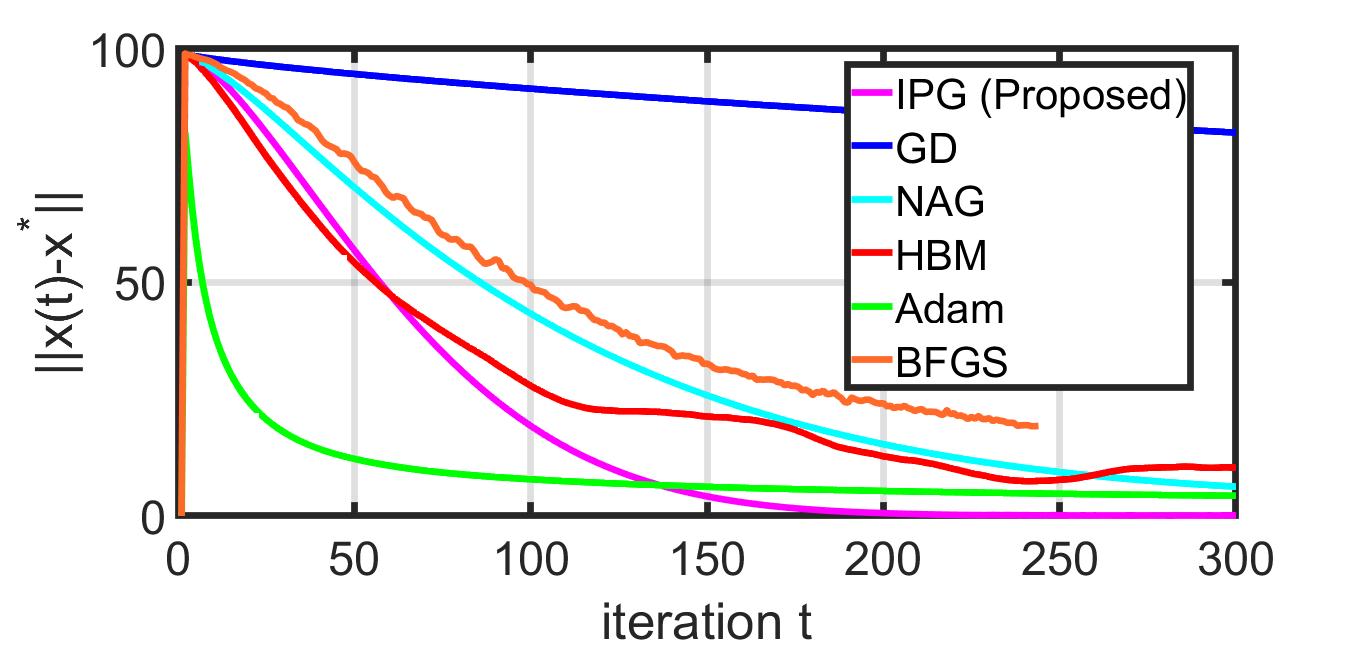}
  \caption{\footnotesize{Estimation error $\norm{x(t)-x^*}$ of different algorithms for the noisy quadratic model of neural network training~\cite{zhang2019algorithmic}. All the algorithms have the same initial estimate $\{x_i(0) \in \mathcal{N}(0,1), ~ i=1,\ldots,d\}$. Other parameters are in the supplemental material.}}
  \label{fig:nqm}
\end{wrapfigure}

We propose an {\em adaptive gradient algorithm} for improving the convergence rate of the distributed gradient-descent method when solving the convex optimization problem~\eqref{eqn:opt_1} in distributed networks. The key concept in our proposed method is {\em iterative pre-conditioning}. The idea of {\em iterative pre-conditioning} has been originally proposed in~\cite{chakrabarti2020iterative}, wherein the server updates the estimate using the sum of the agents' gradient multiplied with a suitable iterative pre-conditioning matrix. However,~\cite{chakrabarti2020iterative} considers only quadratic cost functions. Note that the iterative pre-conditioning in~\cite{chakrabarti2020iterative} does not trivially extend to general cost functions due to non-linearity in the gradients. The proposed algorithm rigorously extends that idea of iterative pre-conditioning to general convex optimization problems~\eqref{eqn:opt_1}. Using real-world datasets, we empirically show that the proposed algorithm \underline{converges in fewer iterations} compared to the aforementioned state-of-the-art methods for solving the distributed convex optimization problem~\eqref{eqn:opt_1}. Besides empirical results, we also present a formal analysis of the proposed algorithm's convergence.

Several research works have used quadratic models for approximating the loss functions of neural networks~\cite{lucas2021analyzing, schaul2013no,jastrzkebski2017three, zhang2019lookahead, wu2018understanding, wen2019empirical}. Quadratic model-based analyses of neural networks have produced vital insights, including learning rate scheduling~\cite{schaul2013no}, and the Bayesian viewpoint of SGD with fixed stepsize~\cite{mandt2017stochastic}. 
The noisy quadratic model (NQM) with carefully chosen model parameters is a proxy for neural network training.
The NQM parameters proposed in~\cite{zhang2019algorithmic} make predictions that are aligned with deep neural networks in realistic experimental settings. The results in~\cite{zhang2019algorithmic} are supported by further evidence from~\cite{lee2020wide}, which rigorously show that quadratic loss function model governs infinite width neural network of arbitrary depth. The theoretical NQM in~\cite{wu2018understanding} correctly captures the short-horizon bias of learning rates in neural network training. Thus, although the general optimization model of neural networks is non-convex, noisy convex quadratic models such as~\cite{zhang2019algorithmic} trim away inessentials features while capturing the key aspects of real neural network training, including generalization performance~\cite{wen2019empirical} and the effects of pre-conditioning on gradients. This motivates us to implement our proposed iterative pre-conditioning scheme on the noisy quadratic model~\cite{zhang2019algorithmic} of neural networks and empirically validating it for solving general non-convex optimization problems, including neural network training.

\subsection{Summary of our contributions}
\label{sub:contri}

% In this paper, we propose a distributed {\em adaptive gradient method}, based on {\em iterative pre-conditioning}, for solving convex optimization problem~\eqref{eqn:opt_1} without requiring the agents to reveal their data points. 
Our key contributions can be summarized as follows.
% We present below a summary of our key contributions. 
% and comparisons with other prominent algorithms applicable to the distributed architecture.

\begin{itemize}
    \item
    % with non-singular Hessians at a minimum point~\eqref{eqn:opt_1}, 
    We formally show that our algorithm converges {\em linearly} for a general class of convex cost functions when the Hessian is non-singular at a solution of~\eqref{eqn:opt_1}. In the special case when the solution of~\eqref{eqn:opt_1} is unique, the convergence of our algorithm is {\em superlinear}. Formal details are presented in Theorem~\ref{thm:conv} and Theorem~\ref{thm:superlinear} in Section~\ref{sub:conv}. 
    
    \item We rigorously show that the convergence rate of our algorithm compares favorably to prominent distributed algorithms, namely the GD, NAG, HBM, and Adam methods.
    
    \begin{itemize}
        \item[-] When the solution~\eqref{eqn:opt_1} is unique, our algorithm converges {\em superlinearly} which is only comparable to the BFGS method~\cite{kelley1999iterative}. The convergence of the other aforementioned algorithms, on the other hand, is only {\em linear}~\cite{fessler2008image,tong2019calibrating,su2014differential}.
        
        \item[-] Moreover, in the general case, our algorithm converges {\em linearly}. On the other hand, the convergence of GD, NAG, HBM, and Adam is {\em sublinear}~\cite{fessler2008image,tong2019calibrating,su2014differential}. Only the BFGS method has {\em superlinear} rate in this case.
    \end{itemize}
    
    \item Though numerical experiments on real-world data analysis problems, we demonstrate the improved convergence rate of our algorithm.
    
    % convergence rate is superior to that of the state-of-the-art methods when distributively solving convex optimization problems. 
    
    \begin{itemize}
        \item[-] The noisy quadratic model in~\cite{zhang2019algorithmic} has been claimed to emulate neural network training. Our empirical study shows that the proposed algorithm converges faster than the aforementioned distributed methods on this model, thereby demonstrating the proposed algorithm's efficiency for distributed solution of non-convex optimization problems. Please refer to Section~\ref{sub:nqm} for further details.
        
        \item[-] Our empirical results for distributed binary logistic regression problem on the ``MNIST'' and ``CIFAR-10'' datasets validate the proposed algorithm's superior convergence rate and comparable test accuracy, particularly under the influence of process noise in the system and the stochastic mini-batch setting. Our algorithm's final estimated training loss is favorable to BFGS and comparable to the other methods. Please refer to Section~\ref{sub:log_reg} for further details.
    \end{itemize}
\end{itemize}

\section{Proposed algorithm: Iteratively Pre-conditioned Gradient-descent (IPG)}
\label{sec:algo}

In this section, we present our algorithm. Our algorithm follows the basic prototype of the distributed gradient-descent method. However, a notable difference is that in our algorithm, the server multiplies the aggregate of the gradients received from the agents by a {\em pre-conditioner} matrix. The server uses the {\em pre-conditioned} aggregates of the agents' gradients to update its current estimates. This technique is commonly known as {\em pre-conditioning} \cite{nocedal2006numerical}. When the aggregate cost $f$ is strongly convex, the best pre-conditioner matrix for the gradient-descent method is the inverse Hessian matrix $\nabla^2 f(x(t))^{-1}$, resulting in Newton's method which converges {\em superlinearly}~\cite{kelley1999iterative}. However, $\nabla^2 f(x(t))^{-1}$ cannot be computed directly in a distributed setting as it requires the agents to send their local Hessian matrices $\nabla^2 f^i(x(t))$ to the server, which may require the agents to share their local data points, such as the quadratic optimization problem~\cite{chakrabarti2020iterative}. Thus, we propose a distributed scheme where the server iteratively updates the pre-conditioner matrix without requiring the agents to share their local Hessian and their local data points in such a way that it linearly converges to $\nabla^2 f(x(t))^{-1}$. Thus, our algorithm eventually converges to Newton's method and has {\em superlinear} convergence rate for strongly convex problems, as shown later in Section~\ref{sub:conv}.

Next, we describe the proposed Iteratively Pre-conditioned Gradient-descent (IPG) algorithm. The algorithm is iterative when in each iteration $t \in \{0, \, 1, \ldots\}$, the server maintains an estimate $x(t)$ of a minimum point~(\ref{eqn:opt_1}), and a pre-conditioner matrix $K(t) \in \R^{d \times d}$. Both the estimate and the pre-conditioner matrix are updated using steps presented below.

{\bf Initialization:} Before starting the iterative process, the server chooses an initial estimate $x(0)$ and a pre-conditioner matrix $K(0)$ from $\R^d$ and  $\R^{d \times d}$, respectively. Further, the server chooses a sequence of non-negative scalar parameters $\{\alpha(t), t\geq 0\}$ and two non-negative scalar real-valued parameters $\delta, ~\beta$. The server broadcasts parameter $\beta$ to all the agents. The specific values of these parameters are presented later in Section~\ref{sub:conv}.

\subsection{Steps in each iteration $t$}
\label{sub:algo_steps}

In each iteration $t$, the algorithm comprises four steps that are executed collaboratively by the server and the agents. 

\begin{itemize}
    \item {\em Step 1:} The server sends the estimate $x(t)$ and the matrix $K(t)$ to each agent $i \in \{1,\ldots,m\}$.
    \item {\em Step 2:} Each agent $i \in \{1,\ldots,m\}$ computes the gradient $g^i(t)$ of its local cost function, defined as
    \begin{align}
        g^i(t) = \nabla f^i(x(t)). \label{eqn:g_i}
    \end{align}
    Let $I$ denote the $(d \times d)$-dimensional identity matrix. Let $e_j$ and $k_j(t)$ denote the $j$-th columns of matrices $I$ and $K(t)$, respectively, so that $K(t) = \begin{bmatrix} k_1(t), \ldots, k_d(t) \end{bmatrix}$.
    
    In the same step, each agent $i$ computes a set of vectors $\left\{R^i_j(t): ~ j = 1, \ldots, \, d \right\}$ such that for each $j$,
    \begin{align}
        R^i_j(t) = \left(\nabla^2 f^i(x(t)) + \left(\dfrac{\beta}{m}\right) I\right) k_j(t) - \left(\frac{1}{m}\right) e_j. \label{eqn:rij}
    \end{align}
    \item {\em Step 3:} Each agent $i$ sends gradient $g^i(t)$ and the set of vectors $\left\{R^i_j(t), ~ j = 1, \ldots, \, d \right\}$ to the server.
    \item {\em Step 4:} The server updates the estimate $x(t)$ to $x(t + 1)$ such that
    \begin{eqnarray}
        x(t+1) = x(t) - \delta K(t) \sum_{i=1}^m g^i(t). \label{eqn:x_update}
    \end{eqnarray}
    The server updates the pre-conditioner matrix $K(t)$ to $K(t+1)$ such that
    \begin{eqnarray}
        k_j(t + 1) = k_j(t) - \alpha(t) \sum_{i=1}^m R^i_j(t), \quad j = 1,...,d. \label{eqn:kcol_update}
    \end{eqnarray}
\end{itemize}

Next, we formally analyze convergence of the proposed IPG algorithm. For analysis, we consider only full batched-data. However, IPG is applicable also to the stochastic setting, as we show later in the experimental results (ref. Section~\ref{sec:exp}).

\subsection{Convergence guarantee}
\label{sub:conv}

% Detailed proof of the results in this subsection is included with the supplemental material.

We make the following assumptions for our theoretical results. Recall that $f$ denotes the aggregate cost function, i.e., $f = \sum_{i=1}^m f^i$, and $x^*$ denotes a minimum point of $f$, defined in~\eqref{eqn:opt_1}.

\begin{assumption} \label{assump_1}
Assume that the minimum of function $f$ exists and is finite, i.e., $\mnorm{\min_{x \in \R^d} f(x)} < \infty$.
% the aggregate cost function $f$ is bounded below by a finite value, i.e., $$ $-\infty < \min_{x \in \R^d} f(x)$.
\end{assumption}

\begin{assumption} \label{assump_2}
Assume that each local cost function $f^i$ is convex and twice continuously differentiable over a convex domain $\D \subseteq \R^d$ containing the set of minimum points $X^*$, defined in~\eqref{eqn:opt_1}. 
\end{assumption}

\begin{assumption} \label{assump_3}
Assume that the Hessian $\nabla^2 f$ is Lipschitz continuous over the domain $\D$ with respect to the $2$-norm with Lipschitz constant $\gamma$. Specifically, for any $x,y \in \D$,
\begin{align*}
    \norm{\nabla^2 f(x) - \nabla^2 f(y)} \leq \gamma \norm{x - y}.
\end{align*}
\end{assumption}

\begin{assumption} \label{assump_4}
Assume that the Hessian $\nabla^2 f$ is non-singular at any minimum point $x^* \in X^*$.
\end{assumption}

Assumption~\ref{assump_2} implies that the aggregate cost function $f = \sum_{i=1}^m f^i$, defined in~\eqref{eqn:opt_1}, is twice continuously differentiable over $\D$. Thus, its gradient, denoted by $\nabla f$ is Lipschitz continuous over $\D$. In other words, there exists a positive real valued constant scalar $l$ such that for any $x,y \in \D$,
\begin{align}
    \norm{\nabla f(x) - \nabla f(y)} \leq l \norm{x - y}. \label{eqn:lip_grad}
\end{align}

{\bf Notation:} To formally state our convergence result we introduce some notation below.
\begin{itemize}
    \item For $\beta > 0$, we define $$K^* = \left(\nabla^2 f(x^*) + \beta I\right)^{-1}.$$ 
    Under Assumption~\ref{assump_2}; the function $f$ is convex, thus $\left(\nabla^2 f(x^*) + \beta I\right)$ is positive definite when $\beta > 0$, and hence $K^*$ is well-defined. 
    \item We let $\eta$ denote the induced $2$-norm of $K^*$, i.e., $\eta = \norm{K^*}$.
    \item For each iteration $t \geq 0$, we define $$\rho(t) = \norm{I - \alpha(t) \left(\nabla^2 f(x(t)) + \beta I\right)}.$$ 
    % We let $\rho = \max_{t \geq 0} \rho(t)$.
    \item We let $\lambda_{max} \left[\cdot \right]$ and $\lambda_{min} \left[\cdot \right]$, respectively denote the largest and smallest eigenvalue of a matrix.
\end{itemize}

Lemma~\ref{lem:rho} below states a preliminary result about the convergence of the sequence of pre-conditioner matrices $\{K(t), t \geq 0\}$ to $K^*$. This lemma is important for our key results presented afterward. Proof of Lemma~\ref{lem:rho} is deferred to Appendix~\ref{prf:lem1}.

\begin{lemma} \label{lem:rho}
Consider the IPG Algorithm with parameters $\beta > 0$ and $\alpha(t)$ subject to 
\begin{align}
    0 < \alpha(t) < \frac{1}{\lambda_{max} \left[\nabla^2 f(x(t))\right] + \beta}, \quad \forall t \geq 0. \label{eqn:alpha_cond_1}
\end{align}
Then, under Assumptions~\ref{assump_1} and~\ref{assump_2}, $\rho(t) \in [0, \, 1)$ for all $t \geq 0$. 
\end{lemma}

Note that under the conditions assumed in Lemma~\ref{lem:rho}, 
\begin{align*}
    \rho := \sup_{t \geq 0} \rho(t) \text{  exists, and is less than } 1.
\end{align*}
We now present below in Theorem~\ref{thm:conv} a key convergence result of our proposed IPG method. Proof of Theorem~\ref{thm:conv} is deferred to Appendix~\ref{prf:thm1}. Recall that $x^*$ denotes a minimum point of $f$, defined in~\eqref{eqn:opt_1}.

\begin{theorem} \label{thm:conv}
Suppose that Assumptions~\ref{assump_1},~\ref{assump_2},~\ref{assump_3} and \ref{assump_4} hold true. Consider the IPG Algorithm with parameters $\beta > 0$, $\delta = 1$ and $\alpha(t)$ satisfying~\eqref{eqn:alpha_cond_1} for all $t \geq 0$. Let the parameter $\beta$, the initial estimate $x(0) \in \D$ and pre-conditioner matrix $K(0)$ be chosen such that
\begin{align}
    \frac{\eta \gamma}{2} \norm{x(0) - x^*} + l \norm{K(0) - K^*} + \eta \beta \leq \frac{1}{2\mu} \label{eqn:init_cond}
\end{align}
where $\mu \in \left(1, ~ \frac{1}{\rho} \right)$ and $\eta = \norm{K^*}$. If for $t \geq 0$,
\begin{align}
    \alpha(t) < \min\left\{\frac{1}{\lambda_{max} \left[\nabla^2 f(x(t))\right] + \beta}, ~ \frac{\mu^t (1-\mu \rho)}{2l(1-(\mu \rho)^{t+1})}\right\}, \label{eqn:alpha_cond}
\end{align}
then we obtain that, for all $t \geq 0$,
\begin{align}
    \norm{x(t+1) - x^*} & \leq \frac{1}{\mu} \norm{x(t) - x^*}. \label{eqn:x_linear}
\end{align}
\end{theorem}

Since $\mu > 1$, Theorem~\ref{thm:conv} implies that the sequence of estimates $\{x(t),t\geq0\}$ locally converges to a solution $x^* \in X^*$ with a {\em linear} convergence rate $\frac{1}{\mu}$. To obtain a simpler condition on $\alpha(t)$, compared to~\eqref{eqn:alpha_cond}, we can use a more conservative upper bound. Specifically, let $\Lambda$ be an upper bound of $\lambda_{max} \left[\nabla^2 f(x(t))\right]$ that holds true for each iteration $t$. From condition~\eqref{eqn:init_cond} in Theorem~\ref{thm:conv}, we have $0 < \mu \rho < 1$ and $\mu > 1$. Thus, 
\[\frac{\mu^t (1-\mu \rho)}{2l(1-(\mu \rho)^{t+1})} > \frac{\mu^t (1-\mu \rho)}{2l}. \]
Ultimately, from above we infer that if 
\[ \alpha(t) < \min\left\{\frac{1}{\Lambda + \beta}, ~ \frac{\mu^t (1-\mu \rho)}{2l}\right\}, \quad \forall t \geq 0,\]
then condition~\eqref{eqn:alpha_cond} is implied, and Theorem~\ref{thm:conv} holds true.
% holds true under a simpler sufficient condition on $\alpha(t)$, instead of~\eqref{eqn:alpha_cond}, 
% Since the function $u(t) = \frac{\mu^t (1-\mu \rho)}{2l(1-(\mu \rho)^{t+1})}$ has a positive minimum at $t^* = \frac{\log\left(-\log(\mu)/\log(\rho) \right)}{\log(\mu \rho)}-1$, $u(t^*)$ can also be substituted in the second term of condition~\eqref{eqn:alpha_cond}. 

\paragraph{The case of strong convexity:} We further show that our algorithm attains {\em superlinear} convergence when the aggregate cost function $f$ is {\em strongly convex} (however, the local costs may only be convex), as stated in the assumption below.

\begin{assumption} \label{assump_5}
Assume that the aggregate cost function $f$ is strongly convex over the domain $\D$.
\end{assumption}

In this case, the Hessian $\nabla^2 f$ is positive definite over the entire domain $\D$. Thus, solution to problem~\eqref{eqn:opt_1} is unique, which we denote by $x^*$. Under Assumption~\ref{assump_5}, we show that IPG method with parameter $\beta = 0$ converges {\em superlinearly}. 
% Specifically, we obtain the following result.

\begin{theorem} \label{thm:superlinear}
Consider the IPG Algorithm presented in Section~\ref{sub:algo_steps}, with parameters $\beta = 0$, $\delta = 1$, and $\alpha(t)$ for each iteration $t \geq 0$. If Assumptions~\ref{assump_1}-\ref{assump_3}, Assumption~\ref{assump_5}, and the conditions~\eqref{eqn:init_cond}-\eqref{eqn:alpha_cond} are satisfied, then the following statement holds true:
\begin{align}
    \lim_{t \to \infty} \frac{\norm{x(t+1) - x^*}}{\norm{x(t) - x^*}} & = 0. \label{eqn:x_superlinear}
\end{align}
\end{theorem} 

Proof of Theorem~\ref{thm:superlinear} is deferred to Appendix~\ref{prf:thm2}.

\section{Experimental results}
\label{sec:exp}

This section presents our results on three different experiments, validating the convergence of our proposed algorithm on real-world problems and its comparison with related methods.  

\subsection{Distributed noisy quadratic model}
\label{sub:nqm}

In the first experiment, we implement our proposed algorithm for distributively solving~\eqref{eqn:opt_1} in the case of the noisy quadratic model (NQM)~\cite{zhang2019algorithmic}, which has been shown to agree with the results of training real neural networks. The noisy quadratic model of neural networks consists of a quadratic aggregate cost function $f(x) =\frac{1}{2} x^T H x, ~ x \in \R^d$ where $H$ is a $(d\times d)$-dimensional diagonal matrix whose $i$-th element in the diagonal is $\frac{1}{i}$, and each gradient query is corrupted with independent and identically distributed noise of zero mean and the diagonal covariance matrix $H$. In the experiments, we choose $d=10^4$, which is also the condition number of the Hessian $H$. The distributed implementation of the aforementioned noisy quadratic model is set up as follows. The data points represented by the rows of the matrix $H$ are divided amongst $m=10$ agents. Thus, each agent $1,\ldots,10$ has a data matrix of dimension $10^3 \times 10^4$. 
Since the Hessian matrix $H$ in the noisy quadratic model is positive definite, the optimization problem~\eqref{eqn:opt_1} has a unique solution $x^* = 0_d$.

%%%%%%%%%%%%%%%%%%%%%%%%%%%%%%%%%%%%%%%%%%%%%%%%%%%%%%%%%%%%%
\begin{table}[htb!]
  \caption{Comparisons between the number of iterations required by different algorithms to attain the specified values for the relative estimation errors $\epsilon_{tol}$.}
  \label{tab:time_comp}
  \centering
  \begin{tabular}{|C{2.2cm}||C{1cm}|C{1cm}|C{1cm}|C{1cm}|C{1cm}|C{1.2cm}|C{1.5cm}|}
    \toprule
    Dataset & $\epsilon_{tol}$ & IPG & GD & NAG & HBM & Adam & BFGS \\
    \midrule
    \midrule
    Noisy Quadratic & $10^{-3}$ & \cellcolor{Gray} $242$ & $>10^4$ & $>10^4$ & $>10^4$ & $>10^4$ & unstable after $245$ \\
    \midrule
    MNIST (deterministic) & $0$ & \cellcolor{Gray} $214$ & $>10^4$ & $486$ & $462$ & $851$ & $39$ \\
    \midrule
    CIFAR-10 (deterministic) & $0$ & \cellcolor{Gray} $196$ & $>10^4$ & $627$ & $634$ & $3750$ & $79$ \\
    \bottomrule
  \end{tabular}
\end{table}
%%%%%%%%%%%%%%%%%%%%%%%%%%%%%%%%%%%%%%%%%%%%%%%%%%%%%%%%%%%%%

We compare the performance of our proposed IPG algorithm with other algorithms when implemented on the above NQM in the distributed architecture. Specifically, these algorithms are the distributed versions of gradient-descent (GD)~\cite{bertsekas1989parallel}, Nesterov's accelerated gradient-descent (NAG)~\cite{nesterov27method}, heavy-ball method (HBM)~\cite{polyak1964some}, adaptive momentum estimation (Adam)~\cite{kingma2014adam}, and BFGS method~\cite{kelley1999iterative}. 

The parameters of the respective distributed algorithms are selected such that each of these methods converges in a fewer number of iterations. Specifically, these parameters are selected as described below. For the GD, NAG, and HBM methods, the algorithm's parameters are set such that each of these methods achieves the optimal (smallest) rate of convergence. The specific definition of these optimal parameters can be found in~\cite{lessard2016analysis}. Since the Hessian of the noisy quadratic model is positive definite, Assumption~\ref{assump_5} holds. So, we set the parameter $\beta=0$ for the IPG method. The optimal convergence rate of the linear version of the proposed IPG method is obtained when $\alpha = \frac{2}{\lambda_1 + \lambda_d}$~\cite{chakrabarti2020iterative}. Here, $\lambda_1=1$ and $\lambda_d=\frac{1}{d}$ respectively denote the largest and the smallest eigenvalue of $H$. We find that the IPG method implemented for the NQM converges fastest when the parameter $\alpha$ is set similarly as $\alpha(t) = \frac{2}{\lambda_1 + \lambda_d}$. The step-size parameter $\alpha(t)$ of Adam is selected from the set $\{c,\frac{c}{\sqrt{t}},\frac{c}{t}\}$ where $c$ is from the set $\{0.01, 0.05, 0.1, 0.5, 1, 2\}$. The other parameters of Adam are set at their usual values of $\beta_1 = 0.9$, $\beta_2 = 0.999$, and $\epsilon = 10^{-8}$. The step-size of the BFGS method is obtained using the backtrack routine. The best parameter combinations from above are reported in the supplemental material.

The initial estimate $x(0)$ for these algorithms is randomly drawn from the standard normal distribution. The initial pre-conditioner matrix $K(0)$ for the IPG algorithm is the zero matrix of dimension $(d \times d)$. The initial approximation of the Hessian matrix $B(0)$ for the BFGS algorithm is the identity matrix of dimension $(d \times d)$.

We compare the number of iterations needed by these algorithms to reach a {\em relative estimation error} defined as $\epsilon_{tol} = \frac{\norm{x(t) - x^*}}{\norm{x(0) - x^*}}$. The results are recorded in Table~\ref{tab:time_comp} and Figure~\ref{fig:nqm}, from which we find that the proposed IPG algorithm converges fastest among the distributed algorithms implemented on the NQM.

%%%%%%%%%%%%%%%%%%%%%%%%%%%%%%%%%%%%%%%%%%%%%%%%%%%%%%%%%%%%%
\begin{table}[h]
  \caption{Comparisons between the number of iterations required by different algorithms to attain the specified values for relative error in estimated cost $\epsilon_{tol}$ when subjected to noise.}
  \label{tab:noise_time}
  \centering
  \begin{tabular}{|C{2.3cm}||C{1.2cm}|C{1cm}|C{1cm}|C{1cm}|C{1cm}|C{1.2cm}|C{1.2cm}|}
    \toprule
    Dataset & $\epsilon_{tol}$ & IPG & GD & NAG & HBM & Adam & BFGS \\
    \midrule
    \midrule
    MNIST (process noise) & $6e{-8}$ & $216$ & $>10^4$ & $>10^4$ & $532$ & $878$ & \cellcolor{Gray} $43$ \\
    \midrule
    MNIST (mini-batch) & $2e{-3}$ & \cellcolor{Gray} $737$ & $>10^4$ & $>10^4$ & $>10^4$ & $>10^4$ & $>10^4$ \\
    \midrule
    CIFAR-10 (process noise) & $4e{-6}$ & \cellcolor{Gray} $289$ & $>10^4$ & $>10^4$ & $350$ & $1191$ & $>10^4$ \\
    \midrule
    CIFAR-10 (mini-batch) & $2e{-3}$ & \cellcolor{Gray} $1960$ & $>10^4$ & $3151$ & $>10^4$ & $>10^4$ & $>10^4$ \\
    \bottomrule
  \end{tabular}
\end{table}

\begin{table}
  \caption{Comparisons between the relative final error $sse$ in estimated cost function obtained by different algorithms when subjected to noise.}
  \label{tab:noise_sse}
  \centering
  \begin{tabular}{|C{2.3cm}||C{1.2cm}|C{1cm}|C{1cm}|C{1cm}|C{1.2cm}|C{1.2cm}|}
    \toprule
    Dataset & IPG & GD & NAG & HBM & Adam & BFGS \\
    \midrule
    \midrule
    MNIST (process noise) & $6e{-8}$ & $6e{-8}$ & $6 e{-8}$ & $6e{-8}$ & \cellcolor{Gray} $3e{-8}$ & $\infty$ \\
    \midrule
    MNIST (mini-batch) & \cellcolor{Gray} $2e{-3}$ & $3e{-3}$ & $5e{-3}$ & $1e{-2}$ & $1e{-1}$ & $1e{-1}$ \\
    \midrule
    CIFAR-10 (process noise) & $4e{-6}$ & $3e{-6}$ & $1e{-5}$ & \cellcolor{Gray} $0$ & $2e{-8}$ & $\infty$ \\
    \midrule
    CIFAR-10 (mini-batch) & $2e{-3}$ & $2e{-2}$ & \cellcolor{Gray} $1e{-3}$ & $8e{-3}$ & $3e{-3}$ & $\infty$ \\
    \bottomrule
  \end{tabular}
\end{table}

%%%%%%%%%%%%%%%%%%%%%%%%%%%%%%%%%%%%%%%%%%%%%%%%%%%%%%%%%%%%%
\begin{figure}
     \centering
     \begin{subfigure}[b]{0.49\textwidth}
         \centering
         \includegraphics[width=\textwidth]{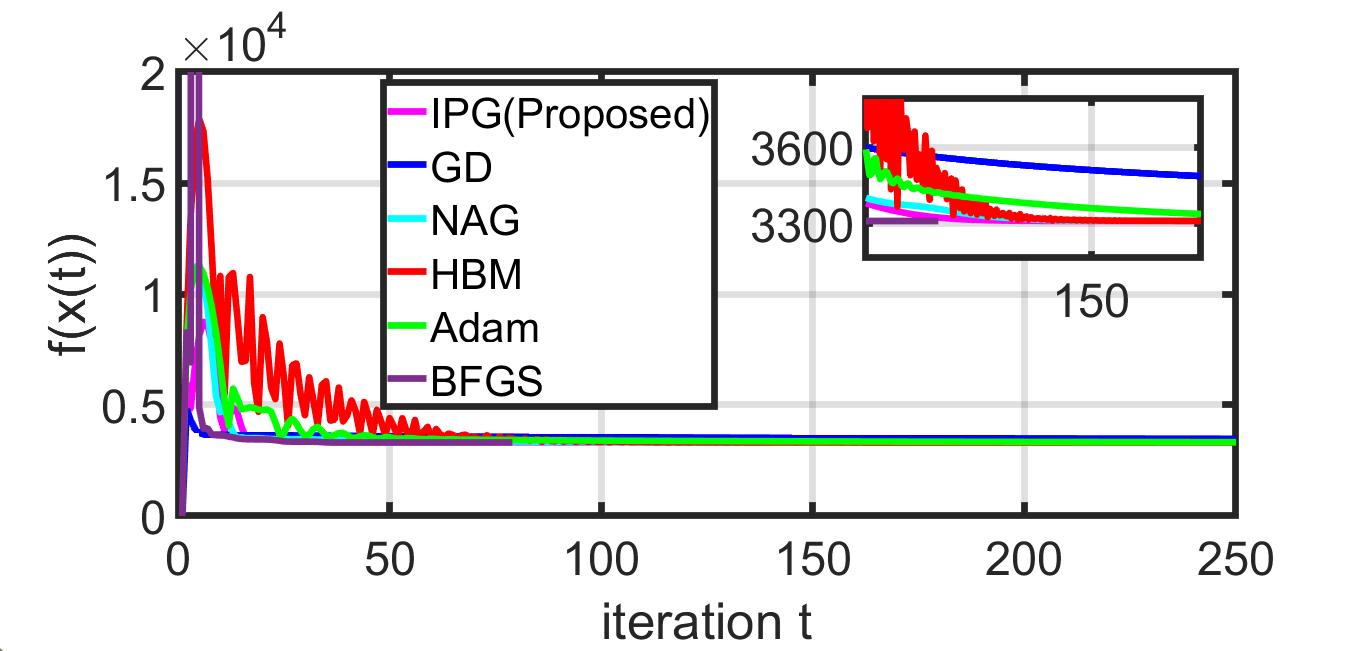}
         \caption{MNIST: full-batch}
         \label{fig:mnist_det}
     \end{subfigure}
     \hfill
     \begin{subfigure}[b]{0.49\textwidth}
         \centering
         \includegraphics[width=\textwidth]{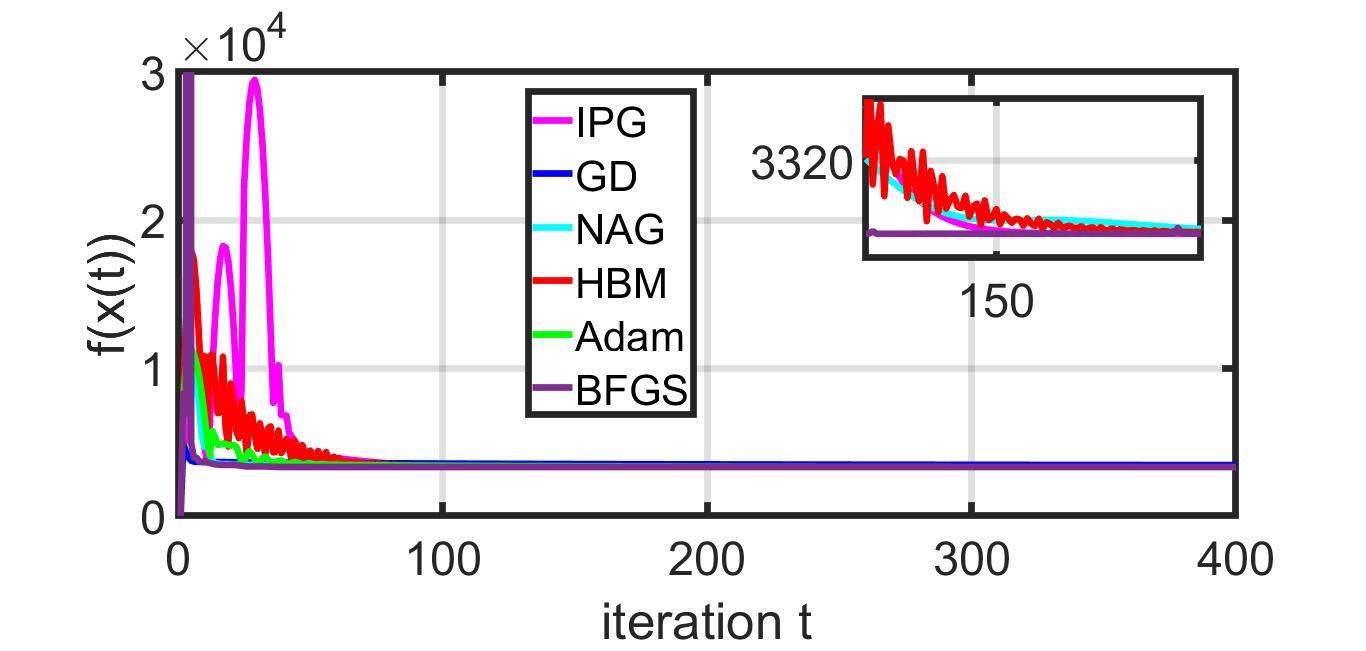}
         \caption{MNIST: full-batch with process noise}
         \label{fig:mnist_noise}
     \end{subfigure}
     \hfill
     \begin{subfigure}[b]{0.49\textwidth}
         \centering
         \includegraphics[width=\textwidth]{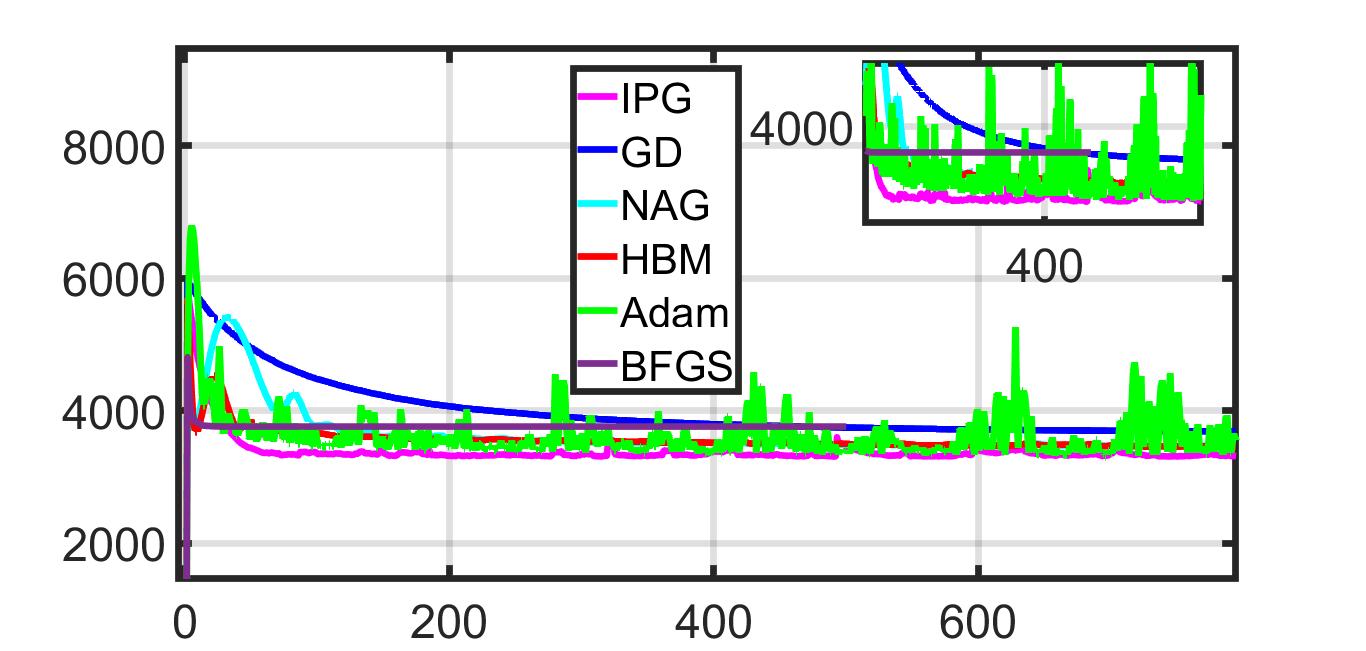}
         \caption{MNIST: mini-batch}
         \label{fig:mnist_mini}
     \end{subfigure}
     \begin{subfigure}[b]{0.49\textwidth}
         \centering
         \includegraphics[width=\textwidth]{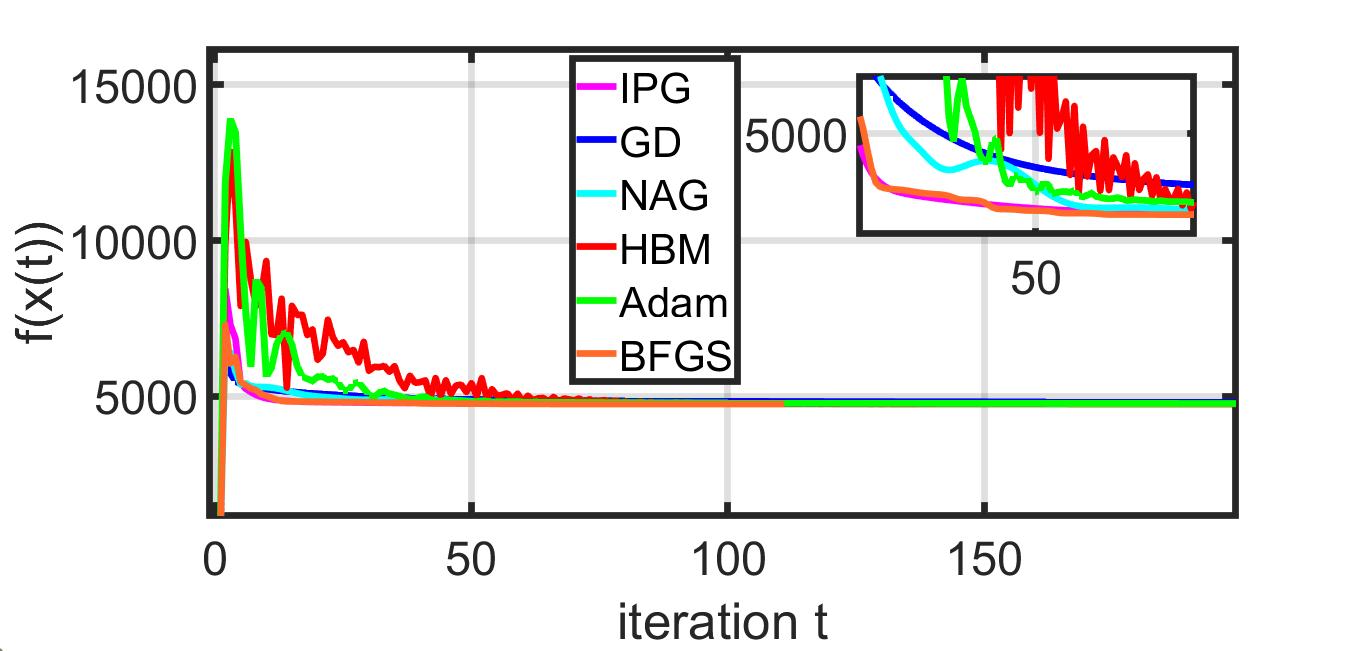}
         \caption{CIFAR-10: full-batch}
         \label{fig:cifar_det}
     \end{subfigure}
     \begin{subfigure}[b]{0.49\textwidth}
         \centering
         \includegraphics[width=\textwidth]{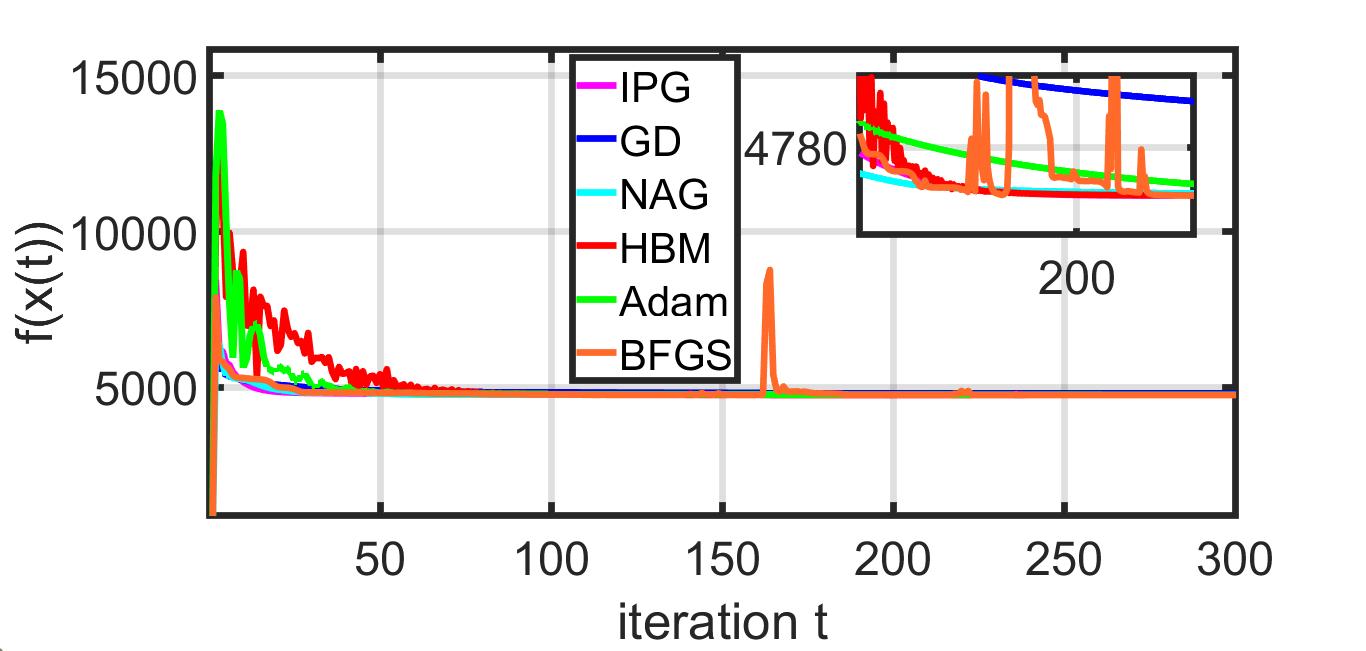}
         \caption{CIFAR-10: full-batch with process noise}
         \label{fig:cifar_noise}
     \end{subfigure}
     \begin{subfigure}[b]{0.49\textwidth}
         \centering
         \includegraphics[width=\textwidth]{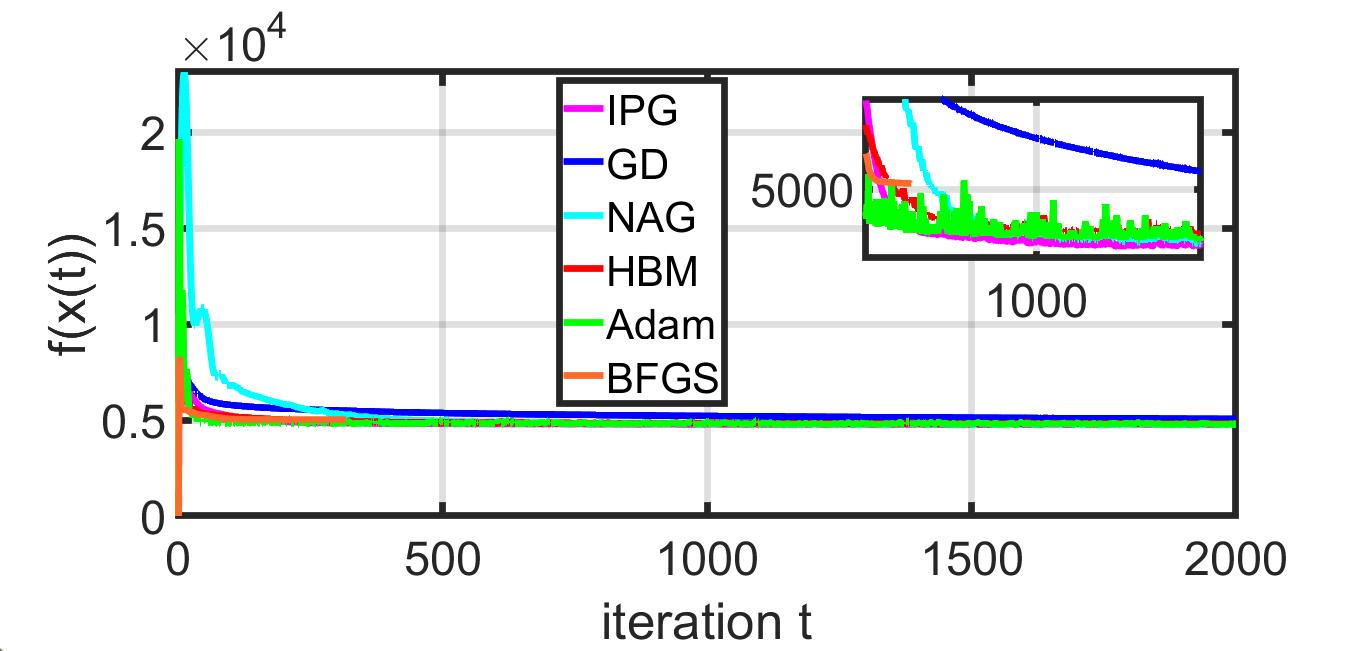}
         \caption{CIFAR-10: mini-batch}
         \label{fig:cifar_mini}
     \end{subfigure}
        \caption{Estimated aggregate cost $f(x(t))$ of different algorithms represented by different colors. For all the algorithms, $\{x_i(0) \in \mathcal{N}(0,0.1), ~ i=1,\ldots,d\}$. Additionally, $K(0) = O_{d \times d}$ for IPG and $B(0) = I_{d \times d}$ for BFGS. The other parameters are enlisted in the supplemental material.}
        \label{fig:log_ref}
\end{figure}
%%%%%%%%%%%%%%%%%%%%%%%%%%%%%%%%%%%%%%%%%%%%%%%%%%%%%%%%%%%%%

\subsection{Distributed logistic regression}
\label{sub:log_reg}

In the second and third experiments, we implement the proposed algorithm for distributively solving the logistic regression problem, respectively on the ``MNIST''~\cite{MNIST} and ``CIFAR-10''~\cite{krizhevsky2009learning} datasets.

From the training examples of the ``MNIST'' dataset, we select $10^4$ arbitrary instances labeled as either the digit one or the digit five. For each instance, we calculate two quantities, namely the average intensity and the average symmetry of the image~\cite{abu2012learning}. Let the column vectors $a_1$ and $a_2$ respectively denote the average intensity and the average symmetry of those $10^4$ instances. These two features are then mapped to a second-order polynomial space. Then, our input data matrix before pre-processing is $\begin{bmatrix} a_1 & a_2 & a_1.^2 & a_1.*a_2 & a_2.^2 \end{bmatrix}$. Here, $(.*)$ represents element-wise multiplication and $(.^2)$ represents element-wise squares. This raw data matrix is then pre-processed as follows. Each column is shifted by the mean value of the corresponding column and then divided by the standard deviation of that column. Finally, a $10^4$-dimensional column vector of unity is appended to this pre-processed matrix. This is our final input matrix $A$ of dimension $(10^4 \times 6)$. The collective data points $(A,B)$ are then distributed among $m=10$ agents, in the manner already described in Section~\ref{sub:nqm}.

Similarly, from the training examples of the ``CIFAR-10'' dataset, we select $10^4$ instances labeled as either ``airplane'' or ``automobile''. For each instance, we calculate six quantities, namely the average intensity and the average symmetry of the image for each of the three pixel-colors. Let the column vector $a_1,\ldots,a_6$ respectively denote these six features of all $10^4$ instances. We apply a feature transform so that our input matrix before pre-processing is $\begin{bmatrix} a_1 & \ldots a_6 & a_1.^2 & \ldots a_6.^2 \end{bmatrix}$. The pre-processing step and splitting the data among agents is the same as described above.

We compare the performance of the proposed IPG method with other algorithms for distributively solving the logistic regression problem on the ``MNIST'' and ``CIFAR-10'' datasets. As in Section~\ref{sub:nqm}, the algorithm parameters are selected such that each of these methods converges in a fewer number of iterations. Specifically, the parameter selection is described below. The best parameter combinations, for which the respective algorithms converge in a fewer number of iterations, are then reported for each dataset in the supplemental material. The algorithms are initialized as described in Section~\ref{sub:nqm}.

\paragraph{IPG:}

The parameter $\alpha(t)$ is selected from the set $\{c\times10^{-3},c\times10^{-4}\}$ where $c$ is from the set $\{1,2,5\}$. The parameter $\delta$ is chosen from the set $\{1,0.1,0.05\}$ and $\beta$ from the set $\{0,0.1,1\}$.

\paragraph{GD:}

The GD algorithm has only one parameter $\alpha$, which is selected from the set $\{c\times10^{-3},c\times10^{-4}\}$ where $c$ is from the set $\{1,2,5\}$.

\paragraph{NAG and HBM:}

The parameter $\alpha$ is selected from the set $\{c\times10^{-3},c\times10^{-4}\}$ where $c$ is from the set $\{1,2,3,5\}$. The parameter $\beta$ is selected from the set $\{0.91,0.92,\ldots,0.99\}$.

\paragraph{Adam:}

The parameters of Adam are selected from the set described in Section~\ref{fig:nqm}.

\paragraph{BFGS:}

The step-size $\alpha(t)$ is either obtained following the backtrack routine or chosen from the set $\{c\times10^{-p}\}$ where $c$ is from the set $\{1,2,5\}$ and $p$ is from $\{2,3,4,5\}$.

For each dataset, we implement the distributed algorithms in three different settings as follows.

\paragraph{Full-batch or deterministic:}

In this setting, we compare the number of iterations needed by different algorithms to reach a minimum point. Let $f^*$ denote the value of aggregate cost function at this minimum point. The results are recorded in Table~\ref{tab:time_comp}, Figure~\ref{fig:mnist_det}, and Figure~\ref{fig:cifar_det}. We observe that only the BFGS algorithm converges in fewer iterations than IPG.

\paragraph{Full-batch with process noise:}

In this setting, we add random process noise to each iterated variable of the respective algorithms. Specifically, the additive noise is {\em uniformly} distributed within $(0,2.3\times 10^{-4})$ and $(0,10^{-4})$, respectively for the MNIST and CIFAR-10 datasets.
We compare (i) the number of iterations needed to reach a {\em relative estimated cost} defined as $\epsilon_{tol} = \frac{f(x(t)) - f^*}{f^*}$, and (ii) the {\em final error in aggregate cost}, defined as $sse = \lim_{t\to \infty} \frac{f(x(t)) - f^*}{f^*}$, obtained by an algorithm. For the criteria-(i), each iterative algorithm is run (ref. Figure~\ref{fig:mnist_noise} and Figure~\ref{fig:cifar_noise}) until its relative estimated cost does not exceed $\epsilon_{tol}$ over a period of $10$ consecutive iterations, and the smallest such iteration is reported in Table~\ref{tab:noise_time}. For the criteria-(ii), each iterative algorithm is run until the changes in its subsequent estimated cost is less than $10^{-4}$ over $50$ consecutive iterations, and the final relative error in estimated cost is reported in Table~\ref{tab:noise_sse}.
For the MNIST dataset, we observe that only the BFGS algorithm reaches $\epsilon_{tol}$ in fewer iterations compared to IPG (ref. Table~\ref{tab:noise_time}). However, the BFGS method diverges after $83$ iterations. The final estimated cost for IPG is comparable to the other algorithms, except for BFGS, which diverges (ref. Table~\ref{tab:noise_sse}). For the CIFAR-10 dataset, the IPG algorithm needs the least iterations (ref. Table~\ref{tab:noise_time}). However, its final cost is slightly larger than HBM and Adam and comparable to GD (ref. Table~\ref{tab:noise_sse}).

%%%%%%%%%%%%%%%%%%%%%%%%%%%%%%%%%%%%%%%%%%%%%%%%%%%%%%%%%%%%%
\begin{table}
  \caption{Comparisons between the test error obtained by training with different algorithms.}
  \label{tab:test}
  \centering
  \begin{tabular}{|C{2.3cm}||C{1.2cm}|C{1.2cm}|C{1.2cm}|C{1.2cm}|C{1.5cm}|C{1.2cm}|}
    \toprule
    Dataset & IPG & GD & NAG & HBM & Adam & BFGS \\
    \midrule
    \midrule
    MNIST (process noise) & $0.13$ & $0.13$ & $0.13$ & $0.13$ & $0.13$ & N/A \\
    \midrule
    MNIST (mini-batch) & $0.14$ & $0.14$ & $0.14$ & $0.14$ & $0.14$ & $0.15$ \\
    \midrule
    CIFAR-10 (process noise) & $0.20$ & $0.20$ & $0.20$ & $0.21$ & $0.21$ & N/A \\
    \midrule
    CIFAR-10 (mini-batch) & $0.21$ & $0.21$ & $0.21$ & $0.21$ & $0.21$ & N/A \\
    \bottomrule
  \end{tabular}
\end{table}
%%%%%%%%%%%%%%%%%%%%%%%%%%%%%%%%%%%%%%%%%%%%%%%%%%%%%%%%%%%%%

\paragraph{Mini-batch:}

Each agent computes its local gradients based on $10$ randomly selected local data points at each iteration in this setting. As there are $m=10$ agents, the server updates the estimate based on an effective batch size of $100$.
We compare the algorithms based on criteria-(i) and -(ii) described in the previous setting. From Table~\ref{tab:noise_time}, the IPG algorithm requires the least number of iterations to reach $\epsilon_{tol}$ for both the datasets. For the MNIST dataset, the IPG algorithm has the smallest final estimated cost (ref. Table~\ref{tab:noise_sse}). For the CIFAR-10 dataset, only NAG has a smaller final cost than IPG (ref. Table~\ref{tab:noise_sse}).

We compare the performance of the algorithms on the respective test datasets (ref. Table~\ref{tab:test}). We observe that the test error of the logistic regression model trained with the proposed IPG algorithm is comparable with the other distributed algorithm, including the GD method. Thus, the IPG algorithm results in faster training without degrading the generalization performance.

\section{Summary}
\label{sec:summary}

We have proposed an adaptive gradient-descent algorithm for distributed learning without requiring the agents to reveal their local data points. The key ingredient of our proposed algorithm is an iterative pre-conditioning technique, which provably improves the convergence rate over the existing adaptive gradient methods in the distributed architecture. We have empirically shown our algorithm's faster convergence, and comparable generalization performance, when compared to the state-of-the-art distributed methods. In this work, we have assumed the distributed network to be synchronous. Our theoretical results have been derived for the full-batch settings and convex cost functions. However, through experiments we have also shown the efficacy of our algorithm in stochastic settings and in training of non-convex neural network via a noisy quadratic model.

\newpage
\bibliography{refs}

\begin{thebibliography}{29}
\providecommand{\natexlab}[1]{#1}
\providecommand{\url}[1]{\texttt{#1}}
\expandafter\ifx\csname urlstyle\endcsname\relax
  \providecommand{\doi}[1]{doi: #1}\else
  \providecommand{\doi}{doi: \begingroup \urlstyle{rm}\Url}\fi

\bibitem[Bertsekas and Tsitsiklis(1989)]{bertsekas1989parallel}
Dimitri~P Bertsekas and John~N Tsitsiklis.
\newblock \emph{Parallel and distributed computation: numerical methods},
  volume~23.
\newblock Prentice hall Englewood Cliffs, NJ, 1989.

\bibitem[Nesterov(1983)]{nesterov27method}
Y~Nesterov.
\newblock A method of solving a convex programming problem with convergence
  rate {$O$}($1/k^2$).
\newblock \emph{Sov. Math. Doklady}, 27\penalty0 (2):\penalty0 372--376, 1983.

\bibitem[Polyak(1964)]{polyak1964some}
Boris~T Polyak.
\newblock Some methods of speeding up the convergence of iteration methods.
\newblock \emph{USSR Computational Mathematics and Mathematical Physics},
  4\penalty0 (5):\penalty0 1--17, 1964.

\bibitem[Kingma and Ba(2014)]{kingma2014adam}
Diederik~P Kingma and Jimmy Ba.
\newblock Adam: A method for stochastic optimization.
\newblock \emph{arXiv preprint arXiv:1412.6980}, 2014.

\bibitem[Kelley(1999)]{kelley1999iterative}
Carl~T Kelley.
\newblock \emph{Iterative methods for optimization}.
\newblock SIAM, 1999.

\bibitem[Duchi et~al.(2011)Duchi, Hazan, and Singer]{duchi2011adaptive}
John Duchi, Elad Hazan, and Yoram Singer.
\newblock Adaptive subgradient methods for online learning and stochastic
  optimization.
\newblock \emph{Journal of machine learning research}, 12\penalty0 (7), 2011.

\bibitem[Zeiler(2012)]{zeiler2012adadelta}
Matthew~D Zeiler.
\newblock Adadelta: An adaptive learning rate method.
\newblock \emph{arXiv preprint arXiv:1212.5701}, 2012.

\bibitem[Dozat(2016)]{dozat2016incorporating}
Timothy Dozat.
\newblock Incorporating nesterov momentum into {A}dam.
\newblock 2016.

\bibitem[Fessler(2021)]{fessler2008image}
Jeffrey~A Fessler.
\newblock Image reconstruction: Algorithms and analysis.
\newblock \url{http://web.eecs.umich.edu/~fessler/book/c-opt.pdf}, 2021.
\newblock [Online book draft; accessed 27-May-2021].

\bibitem[Wu et~al.(2016)Wu, Schuster, Chen, Le, Norouzi, Macherey, Krikun, Cao,
  Gao, Macherey, et~al.]{wu2016google}
Yonghui Wu, Mike Schuster, Zhifeng Chen, Quoc~V Le, Mohammad Norouzi, Wolfgang
  Macherey, Maxim Krikun, Yuan Cao, Qin Gao, Klaus Macherey, et~al.
\newblock Google's neural machine translation system: Bridging the gap between
  human and machine translation.
\newblock \emph{arXiv preprint arXiv:1609.08144}, 2016.

\bibitem[Radford et~al.(2015)Radford, Metz, and
  Chintala]{radford2015unsupervised}
Alec Radford, Luke Metz, and Soumith Chintala.
\newblock Unsupervised representation learning with deep convolutional
  generative adversarial networks.
\newblock \emph{arXiv preprint arXiv:1511.06434}, 2015.

\bibitem[Peters et~al.(2018)Peters, Neumann, Iyyer, Gardner, Clark, Lee, and
  Zettlemoyer]{peters2018deep}
Matthew~E Peters, Mark Neumann, Mohit Iyyer, Matt Gardner, Christopher Clark,
  Kenton Lee, and Luke Zettlemoyer.
\newblock Deep contextualized word representations.
\newblock \emph{arXiv preprint arXiv:1802.05365}, 2018.

\bibitem[Tong et~al.(2019)Tong, Liang, and Bi]{tong2019calibrating}
Qianqian Tong, Guannan Liang, and Jinbo Bi.
\newblock Calibrating the adaptive learning rate to improve convergence of
  {ADAM}.
\newblock \emph{arXiv preprint arXiv:1908.00700}, 2019.

\bibitem[Su et~al.(2014)Su, Boyd, and Candes]{su2014differential}
Weijie Su, Stephen Boyd, and Emmanuel Candes.
\newblock A differential equation for modeling {N}esterov’s accelerated
  gradient method: Theory and insights.
\newblock \emph{Advances in neural information processing systems},
  27:\penalty0 2510--2518, 2014.

\bibitem[Zhang et~al.(2019{\natexlab{a}})Zhang, Li, Nado, Martens, Sachdeva,
  Dahl, Shallue, and Grosse]{zhang2019algorithmic}
Guodong Zhang, Lala Li, Zachary Nado, James Martens, Sushant Sachdeva, George~E
  Dahl, Christopher~J Shallue, and Roger Grosse.
\newblock Which algorithmic choices matter at which batch sizes? {I}nsights
  from a noisy quadratic model.
\newblock \emph{arXiv preprint arXiv:1907.04164}, 2019{\natexlab{a}}.

\bibitem[Chakrabarti et~al.(2020)Chakrabarti, Gupta, and
  Chopra]{chakrabarti2020iterative}
Kushal Chakrabarti, Nirupam Gupta, and Nikhil Chopra.
\newblock Iterative pre-conditioning to expedite the gradient-descent method.
\newblock In \emph{2020 American Control Conference (ACC)}, pages 3977--3982.
  IEEE, 2020.

\bibitem[Lucas et~al.(2021)Lucas, Bae, Zhang, Fort, Zemel, and
  Grosse]{lucas2021analyzing}
James Lucas, Juhan Bae, Michael~R Zhang, Stanislav Fort, Richard Zemel, and
  Roger Grosse.
\newblock Analyzing monotonic linear interpolation in neural network loss
  landscapes.
\newblock \emph{arXiv preprint arXiv:2104.11044}, 2021.

\bibitem[Schaul et~al.(2013)Schaul, Zhang, and LeCun]{schaul2013no}
Tom Schaul, Sixin Zhang, and Yann LeCun.
\newblock No more pesky learning rates.
\newblock In \emph{International Conference on Machine Learning}, pages
  343--351. PMLR, 2013.

\bibitem[Jastrz{\k{e}}bski et~al.(2017)Jastrz{\k{e}}bski, Kenton, Arpit,
  Ballas, Fischer, Bengio, and Storkey]{jastrzkebski2017three}
Stanis{\l}aw Jastrz{\k{e}}bski, Zachary Kenton, Devansh Arpit, Nicolas Ballas,
  Asja Fischer, Yoshua Bengio, and Amos Storkey.
\newblock Three factors influencing minima in {SGD}.
\newblock \emph{arXiv preprint arXiv:1711.04623}, 2017.

\bibitem[Zhang et~al.(2019{\natexlab{b}})Zhang, Lucas, Hinton, and
  Ba]{zhang2019lookahead}
Michael~R Zhang, James Lucas, Geoffrey Hinton, and Jimmy Ba.
\newblock Lookahead optimizer: k steps forward, 1 step back.
\newblock \emph{arXiv preprint arXiv:1907.08610}, 2019{\natexlab{b}}.

\bibitem[Wu et~al.(2018)Wu, Ren, Liao, and Grosse]{wu2018understanding}
Yuhuai Wu, Mengye Ren, Renjie Liao, and Roger Grosse.
\newblock Understanding short-horizon bias in stochastic meta-optimization.
\newblock \emph{arXiv preprint arXiv:1803.02021}, 2018.

\bibitem[Wen et~al.(2019)Wen, Luk, Gazeau, Zhang, Chan, and
  Ba]{wen2019empirical}
Yeming Wen, Kevin Luk, Maxime Gazeau, Guodong Zhang, Harris Chan, and Jimmy Ba.
\newblock An empirical study of large-batch stochastic gradient descent with
  structured covariance noise.
\newblock \emph{arXiv e-prints}, pages arXiv--1902, 2019.

\bibitem[Mandt et~al.(2017)Mandt, Hoffman, and Blei]{mandt2017stochastic}
Stephan Mandt, Matthew~D Hoffman, and David~M Blei.
\newblock Stochastic gradient descent as approximate {B}ayesian inference.
\newblock \emph{arXiv preprint arXiv:1704.04289}, 2017.

\bibitem[Lee et~al.(2020)Lee, Xiao, Schoenholz, Bahri, Novak, Sohl-Dickstein,
  and Pennington]{lee2020wide}
Jaehoon Lee, Lechao Xiao, Samuel~S Schoenholz, Yasaman Bahri, Roman Novak,
  Jascha Sohl-Dickstein, and Jeffrey Pennington.
\newblock Wide neural networks of any depth evolve as linear models under
  gradient descent.
\newblock \emph{Journal of Statistical Mechanics: Theory and Experiment},
  2020\penalty0 (12):\penalty0 124002, 2020.

\bibitem[Nocedal and Wright(2006)]{nocedal2006numerical}
Jorge Nocedal and Stephen Wright.
\newblock \emph{Numerical optimization}.
\newblock Springer Science \& Business Media, 2006.

\bibitem[Lessard et~al.(2016)Lessard, Recht, and Packard]{lessard2016analysis}
Laurent Lessard, Benjamin Recht, and Andrew Packard.
\newblock Analysis and design of optimization algorithms via integral quadratic
  constraints.
\newblock \emph{SIAM Journal on Optimization}, 26\penalty0 (1):\penalty0
  57--95, 2016.

\bibitem[MNI(2018)]{MNIST}
{MNIST} in {CSV}.
\newblock \url{https://www.kaggle.com/oddrationale/mnist-in-csv}, 2018.
\newblock Accessed: 22-May-2021.

\bibitem[Krizhevsky et~al.(2009)Krizhevsky, Hinton,
  et~al.]{krizhevsky2009learning}
Alex Krizhevsky, Geoffrey Hinton, et~al.
\newblock Learning multiple layers of features from tiny images.
\newblock 2009.

\bibitem[Abu-Mostafa et~al.(2012)Abu-Mostafa, Magdon-Ismail, and
  Lin]{abu2012learning}
Yaser~S Abu-Mostafa, Malik Magdon-Ismail, and Hsuan-Tien Lin.
\newblock \emph{Learning from data}, volume~4.
\newblock AMLBook New York, NY, USA, 2012.

\end{thebibliography}

\newpage
\appendix
\section{Supplemental material}

\subsection{Algorithm parameters used in the experiments}
% \FloatBarrier
\begin{table}[h]
\caption{The parameters used in different algorithms.*}
  \label{tab:parameters}
  \centering
  \begin{tabular}{|m{1.5cm}||m{1.8cm}|m{0.7cm}|m{1.5cm}|m{1.5cm}|m{2cm}|m{2cm}|}
    \toprule
    \multirow{2}{*}{Dataset} & IPG & GD & NAG & HBM & Adam & BFGS \\
    % \midrule
    \cline{2-7}
    ~ & $\alpha(t), \delta, \beta$ & $\alpha$ & $\alpha, \beta$ & $\alpha, \, \beta$ & $\alpha(t), \, \beta_1, \, \beta_2$, $\epsilon$ & $\alpha(t)$ \\
    \midrule
    \midrule
    Noisy Quadratic & $1.99, 1, 0$ & $1.99$ & $1.33, 0.97$ & $3.92, 0.96$ & $\frac{1}{t},  0.9, 0.999$, $1e{-8}$ & according to backtrack~\cite{kelley1999iterative} \\
    \midrule
    MNIST (deterministic) & $5e{-4}, 1,  0$ & $5e{-4}$ & $5e{-4}, 0.97$ & $1e{-3}, 0.94$ & $2, 0.9, 0.999$, $1e{-8}$ & $1e{-3}$
    \\
    \midrule
    MNIST (process noise) & $5e{-4}, 1, 0$ & $5e{-4}$ & $5e{-4}, 0.97$ & $1e{-3}, 0.94$ & $2,0.9, 0.999$, $1e{-8}$ & $1e{-3}$
    \\
    \midrule
    MNIST (mini-batch) & $1e{-4}, 0.05,  1$ & $1e{-4}$ & $5e{-4}, 0.97$ & $1e{-3},  0.95$ & $1, 0.9, 0.999$, $1e{-8}$ & $0.05$
    \\
    \midrule
    CIFAR-10 (deterministic) & $2e{-4},  1,  0$ & $2e{-4}$ & $1e{-4},0.95$ & $3e{-4}, 0.94$ & $\frac{1}{\sqrt{t}}, 0.9, 0.999$, $1e{-8}$ & $2e{-4}$\\
    \midrule
    CIFAR-10 (process noise) & $2e{-4}, 0.05,  0$ & $2e{-4}$ & $1e{-4}, 0.93$ & $3e{-4}, 0.94$ & $0.1, 0.9, 0.999$, $1e{-8}$ & $1e{-5}$
    \\
    \midrule
    CIFAR-10 (mini-batch) & $1e{-3},  0.05, 0$ & $2e{-4}$ & $2e{-4}, 0.95$ & $2e{-4}, 0.92$ & $\frac{1}{\sqrt{t}}, 0.9, 0.999$, $1e{-8}$ & according to backtrack
    \\
    \bottomrule
  \end{tabular}
\end{table}
% \FloatBarrier

\textit{* The argument behind selection of these specific parameter values has been described in Section~\ref{sub:nqm}, for the noisy quadratic model, and Section~\ref{sub:log_reg}, for the MNIST and CIFAR-10 datasets.}

\subsection{Proof of Lemma~\ref{lem:rho}}
\label{prf:lem1}

In this section, we prove the result in Lemma~\ref{lem:rho}. 

Recall the definition of $K^*$ from Section~\ref{sub:conv}. For each iteration $t \geq 0$, define the matrix
\begin{align}
    \widetilde{K}(t) = K(t) - K^*. \label{eqn:ktildedef}
\end{align}
Consider an arbitrary iteration $t \geq 0$. Upon substituting from~\eqref{eqn:rij} and~\eqref{eqn:kcol_update} in the definition~\eqref{eqn:ktildedef},
\begin{align*}
    \widetilde{K}(t+1) = & K(t) - \alpha(t) \left(\left(\nabla^2 f(x(t)) + \beta I\right) K(t) -I \right) - K^* \\
    = & K(t) - K^* - \alpha(t) \left(\left(\nabla^2 f(x(t)) + \beta I\right) K(t) - \left(\nabla^2 f(x(t)) + \beta I\right) K^* \right) \\
    & - \alpha(t) \left( \left(\nabla^2 f(x(t)) + \beta I\right) K^* - \left(\nabla^2 f(x^*) + \beta I\right) \left(\nabla^2 f(x^*) + \beta I\right)^{-1} \right).
\end{align*}
Upon substituting above from~\eqref{eqn:ktildedef} and the definition of $K^*$,
\begin{align*}
    \widetilde{K}(t+1) = & \left(I - \alpha(t) \left(\nabla^2 f(x(t)) + \beta I\right) \right) \widetilde{K}(t) - \alpha(t) \left(\nabla^2 f(x(t)) - \nabla^2 f(x^*) \right) K^*.
\end{align*}
Applying Cauchy-Schwartz inequality above,
\begin{align}
    \norm{\widetilde{K}(t+1)} \leq & \norm{I - \alpha(t) \left(\nabla^2 f(x(t)) + \beta I\right)} \norm{\widetilde{K}(t)} + \alpha(t) \norm{\nabla^2 f(x(t)) - \nabla^2 f(x^*)} \norm{K^*}. \label{eqn:norm_kt}
\end{align}
Under Assumption~\ref{assump_2}, each local cost function $f^i$ is convex, implying that the aggregate cost function $f = \sum_{i=1}^m f^i$ is convex. Thus, $\left(\nabla^2 f(x(t)) + \beta I\right)$ is positive definite for $\beta > 0$. It follows that, if $\alpha(t) < \frac{1}{\lambda_{max} \left[\nabla^2 f(x(t))\right] + \beta}$ then $\rho(t) = \norm{I - \alpha(t) \left(\nabla^2 f(x(t)) + \beta I\right)} < 1$. This proves the lemma.

\subsection{Proof of Theorem~\ref{thm:conv}}
\label{prf:thm1}

In this section, we present a formal proof of our main result in Theorem~\ref{thm:conv}. We proceed with the proof in three steps as follows. 

Before presenting the proof, we define the following notation. Recall the definition of the minimum points $X^*$ in~\eqref{eqn:opt_1}. For a minimum points $x^* \in X^*$ and each iteration $t \geq 0$, define the estimation error
\begin{align}
    z(t) = x(t) - x^*. \label{eqn:zdef}
\end{align}

\paragraph{Step 1:}

In this step, we derive an upper bound of $\norm{\widetilde{K}(t+1)}$, provided with the iterations from $0$ to $t$. In this step, we use~\eqref{eqn:norm_kt} from the proof of Lemma~\ref{lem:rho}.

Consider an arbitrary iteration $t \geq 0$. From Section~\ref{sub:conv}, recall that $\eta$ denotes the induced $2$-norm of the matrix $K^*$. Then,
\begin{align}
    \eta = \norm{K^*} = \norm{\left(\nabla^2 f(x^*) + \beta I\right)^{-1}} = \frac{1}{\lambda_{min} \left[\nabla^2 f(x^*)\right] + \beta}. \label{eqn:eta}
\end{align}
Under Assumption~\ref{assump_2} and~\ref{assump_4}, $\nabla^2 f(x^*)$ is positive definite. Then, $\lambda_{min} \left[\nabla^2 f(x^*)\right] > 0$, and $\eta < \frac{1}{\beta}$. 
From Section~\ref{sub:conv}, recall the definition $\rho = \sup_{t \geq 0} \rho(t)$. From Lemma~\ref{lem:rho} we have $0 \leq \rho < 1$.
Upon substituting from~\eqref{eqn:eta} and the definition of $\rho$ in~\eqref{eqn:norm_kt},
\begin{align}
    \norm{\widetilde{K}(t+1)} \leq & \rho \norm{\widetilde{K}(t)} + \eta ~ \alpha(t) \norm{\nabla^2 f(x(t)) - \nabla^2 f(x^*)}. \label{eqn:norm_kt_hess}
\end{align}
Under Assumption~\ref{assump_3}, $\norm{\nabla^2 f(x(t)) - \nabla^2 f(x^*)} \leq \gamma \norm{x(t) - x^*}$. Upon substituting above from the definition~\eqref{eqn:zdef}, $\norm{\nabla^2 f(x(t)) - \nabla^2 f(x^*)} \leq \gamma \norm{z(t)}$. Upon substituting from above in~\eqref{eqn:norm_kt_hess},
\begin{align*}
    \norm{\widetilde{K}(t+1)} \leq & \rho \norm{\widetilde{K}(t)} + \eta \gamma ~ \alpha(t) \norm{z(t)}.
\end{align*}
Iterating the above from $t$ to $0$ we have
\begin{align}
    \norm{\widetilde{K}(t+1)} \leq & \rho^{t+1} \norm{\widetilde{K}(0)} + \eta \gamma ~ \alpha(t) \left( \norm{z(t)} + \rho \norm{z(t-1)} + \ldots + \rho^t \norm{z(0)} \right). \label{eqn:kt_conv}
\end{align}

\paragraph{Step 2:}

In this step, we derive an upper bound on the estimation error $\norm{z(t+1)}$, provided iteration $t$.

Upon substituting from~\eqref{eqn:x_update} in~\eqref{eqn:zdef}, with the parameter $\delta = 1$,
\begin{align*}
    z(t+1) = & z(t) - K(t) \sum_{i=1}^m g^i(t).
\end{align*}
Upon substituting above from~\eqref{eqn:g_i} and the definition of aggregate cost function $f = \sum_{i=1}^m f^i$,
\begin{align*}
    z(t+1) = & z(t) - K(t) \nabla f(x(t)).
\end{align*}
Upon substituting above from the definition of $\widetilde{K}(t)$ in ~\eqref{eqn:ktildedef},
\begin{align*}
    z(t+1) = & z(t) - K^* \nabla f(x(t)) - \widetilde{K}(t) \nabla f(t) \\
    = & -K^* \left(\nabla f(x(t)) - \left(K^*\right)^{-1} z(t) \right) - \widetilde{K}(t) \nabla f(x(t)).
\end{align*}
Since $x^*$ is a minimum point of the aggregate cost function $f$ in~\eqref{eqn:opt_1}, the first order necessary condition states that $\nabla f(x^*) = 0_d$. Here, $0_d$ denotes the $d$-dimensional zero vector. Thus, the above equation can be rewritten as
\begin{align}
    z(t+1)
    = & -K^* \left(\nabla f(x(t)) - \nabla f(x^*) - \left(K^*\right)^{-1} z(t) \right) - \widetilde{K}(t) \nabla f(x(t)). \label{eqn:zt}
\end{align}
Using the fundamental theorem of calculus~\cite{kelley1999iterative},
\begin{align*}
    \nabla f(x(t)) - \nabla f(x^*) = (x(t)-x^*) \int_0^1 \nabla^2 f(y x(t) + (1-y)x^*) dy .
\end{align*}
From above and the definition of $K^* = \left(\nabla^2 f(x^*) + \beta I\right)^{-1}$ (see Section~\ref{sub:conv}), we have
\begin{align*}
    & \nabla f(x(t)) - \nabla f(x^*) - \left(K^*\right)^{-1} z(t) \\
    & =  (x(t)-x^*) \int_0^1 \nabla^2 f(y x(t) + (1-y)x^*) dy  - \left( \nabla^2 f(x^*) + \beta I \right) z(t) \\
    & = (x(t)-x^*) \int_0^1 \left(\nabla^2 f(y x(t) + (1-y)x^*) - \nabla^2 f(x^*) \right) dy  - \beta z(t).
\end{align*}
From~\eqref{eqn:zdef}, recall that $z(t) = x(t) - x^*$. Thus, from above we obtain that 
\begin{align*}
    & \norm{\nabla f(x(t)) - \nabla f(x^*) - \left(K^*\right)^{-1} z(t)} \\
    & \leq \norm{z(t)} \int_0^1 \norm{\nabla^2 f(y x(t) + (1-y)x^*) - \nabla^2 f(x^*)} dy  + \beta \norm{z(t)}.
\end{align*}
Under Assumption~\ref{assump_3}, $\norm{\nabla^2 f(y x(t) + (1-y)x^*) - \nabla^2 f(x^*)} \leq \gamma (1-y) \norm{z(t)}$. Thus, 
\begin{align*}
    \norm{\nabla f(x(t)) - \nabla f(x^*) - \left(K^*\right)^{-1} z(t)} & \leq \gamma \norm{z(t)}^2 \int_0^1 (1-y) dy  + \beta \norm{z(t)} \\ & = \frac{\gamma}{2}\norm{z(t)}^2 + \beta \norm{z(t)}.
\end{align*}
Upon using Cauchy-Schwartz inequality in~\eqref{eqn:zt}, and then substituting from above, we obtain that
\begin{align*}
    \norm{z(t+1)} \leq & \eta \frac{\gamma}{2}\norm{z(t)}^2 + \eta \beta \norm{z(t)} + \norm{\widetilde{K}(t)} \norm{\nabla f(x(t))}.
\end{align*}
Since $\nabla f(x^*) = 0_d$, the above can be rewritten as
\begin{align*}
    \norm{z(t+1)} \leq & \eta \frac{\gamma}{2}\norm{z(t)}^2 + \eta \beta \norm{z(t)} + \norm{\widetilde{K}(t)} \norm{\nabla f(x(t)) - \nabla f(x^*)}.
\end{align*}
Upon using the Lipschitz property~\eqref{eqn:lip_grad} in above,
\begin{align}
    \norm{z(t+1)} \leq & \eta \frac{\gamma}{2}\norm{z(t)}^2 + \eta \beta \norm{z(t)} + l \norm{\widetilde{K}(t)} \norm{z(t)}. \label{eqn:zt_kt}
\end{align}
Upon substituting above from~\eqref{eqn:kt_conv} in Step 1, we have
\begin{align}
    & \norm{z(t+1)} \leq \eta \frac{\gamma}{2}\norm{z(t)}^2 + \eta \beta \norm{z(t)} \nonumber \\
    & + l \norm{z(t)} \left( \rho^{t} \norm{\widetilde{K}(0)} + \eta \gamma ~ \alpha(t-1) \left( \norm{z(t-1)} + \rho \norm{z(t-2)} + \ldots + \rho^{t-1} \norm{z(0)} \right) \right). \label{eqn:zt_full}
\end{align}

\paragraph{Step 3:}

In this final step of the proof, we prove~\eqref{eqn:x_linear} which is a direct result of the following claim.

\begin{claim}
Given the conditions of Theorem~\ref{thm:conv}, for each iteration $t \geq 0$ the following statement holds true:
\begin{align}
    \norm{z(t+1)} \leq \frac{1}{\mu} \norm{z(t)} \, \text{ and } \norm{z(t)} < \frac{1}{\mu \eta \gamma}. \label{eqn:clm}
\end{align}
\end{claim}

\begin{proof}
We proof the aforementioned claim by the principle of induction. First, we show that the claim is true for iteration $t=0$. At $t=0$, from~\eqref{eqn:zt_kt} we have
\begin{align*}
    \norm{z(1)} \leq & \norm{z(0)} \left(\eta \frac{\gamma}{2}\norm{z(0)} + \eta \beta + l \norm{\widetilde{K}(0)} \right).
\end{align*}
Upon substituting above from the condition~\eqref{eqn:init_cond},
\begin{align*}
    \norm{z(1)} \leq & \frac{1}{2\mu}\norm{z(0)}.
\end{align*}
Since $\mu > 1$, from above we have
\begin{align*}
    \norm{z(1)} < \frac{1}{\mu}\norm{z(0)}.
\end{align*}
Moreover,~\eqref{eqn:init_cond} implies that $\norm{z(0)} < \frac{1}{\mu \eta \gamma}$. Thus, the claim holds for $t=0$. Now, we assume that the claim is true for the iterations from $0$ to $t$. We need to show that the claim is also true for the iteration $t+1$. From the above assumption we have
\begin{align}
    \norm{z(t+1)} \leq \frac{1}{\mu} \norm{z(t)} \leq \frac{1}{\mu^2} \norm{z(t-1)} \leq \ldots \leq \frac{1}{\mu^{t+1}} \norm{z(0)} < \frac{1}{\mu^{t+1}} \frac{1}{\mu \eta \gamma}. \label{eqn:zt_contract}
\end{align}
Since $\mu > 1$, the above implies that $\norm{z(t+1)} < \frac{1}{\mu \eta \gamma}$.
Now, consider the expression $\left( \norm{z(t)} + \rho \norm{z(t)} + \ldots + \rho^{t} \norm{z(0)} \right)$ in the R.H.S. of~\eqref{eqn:zt_full} for iteration $t+1$. Upon substituting from~\eqref{eqn:zt_contract},
\begin{align*}
    \norm{z(t)} + \rho \norm{z(t)} + \ldots + \rho^{t} \norm{z(0)} \leq \norm{z(0)} \left( \frac{1}{\mu^t} + \frac{\rho}{\mu^{t-1}} + \ldots + \rho^t \right) = \norm{z(0)} \frac{1-(\mu \rho)^{t+1}}{\mu^t (1-\mu \rho)}.
\end{align*}
From the condition of Theorem~\ref{thm:conv} we have that $\mu \rho < 1$. Upon substituting from above in~\eqref{eqn:zt_full} for iteration $t+1$,
\begin{align}
    & \norm{z(t+2)} \leq \norm{z(t+1)} \left(\eta \frac{\gamma}{2}\norm{z(t+1)} + \eta \beta + l \rho^{t+1} \norm{\widetilde{K}(0)} + \eta \gamma l ~ \alpha(t) \norm{z(0)} \frac{1-(\mu \rho)^{t+1}}{\mu^t (1-\mu \rho)} \right). \label{eqn:zt2}
\end{align}
From Step 1 we have $\rho < 1$. Thus, $ l \rho^{t+1} \norm{\widetilde{K}(0)} < l \norm{\widetilde{K}(0)}$. Additionally, if $\alpha(t) < \frac{\mu^t (1-\mu \rho)}{2l(1-(\mu \rho)^{t+1}}$, then 
$\eta \gamma l ~ \alpha(t) \norm{z(0)} \frac{1-(\mu \rho)^{t+1}}{\mu^t (1-\mu \rho)} < \eta \frac{\gamma}{2} \norm{z(0)}$. Using condition~\eqref{eqn:init_cond} then we have
\begin{align}
    \eta \beta + l \rho^{t+1} \norm{\widetilde{K}(0)} + \eta \gamma l ~ \alpha(t) \norm{z(0)} \frac{1-(\mu \rho)^{t+1}}{\mu^t (1-\mu \rho)} < \frac{1}{2\mu}. \label{eqn:half_mu}
\end{align}
We have shown above that $\norm{z(t+1)} < \frac{1}{\mu \eta \gamma}$, which implies that $\eta \frac{\gamma}{2}\norm{z(t+1)} < \frac{1}{2\mu}$. Upon substituting from above and~\eqref{eqn:half_mu} in~\eqref{eqn:zt2},
\begin{align*}
    \norm{z(t+2)} \leq \frac{1}{\mu} \norm{z(t+1)}.
\end{align*}
Thus, the claim holds for iteration $t+1$. Due to the principle of induction, the proof of the claim is complete.
\end{proof}
Equation~\eqref{eqn:x_linear} directly follows from the aforementioned claim, and concludes the proof of Theorem~\ref{thm:conv}.

\subsection{Proof of Theorem~\ref{thm:superlinear}}
\label{prf:thm2}

In this section, we present the proof of Theorem~\ref{thm:superlinear}. For this proof, we borrow the results~\eqref{eqn:kt_conv} and~\eqref{eqn:zt_kt} from the proof of Theorem~\ref{thm:conv}.

Under Assumption~\ref{assump_5}, $\nabla^2 f(x(t))$ is positive definite for $x(t) \in \D$. So, Assumption~\ref{assump_5} is a special case of Assumption~\ref{assump_4}. Moreover, since $\nabla^2 f(x(t))$ is positive definite for $x(t) \in \D$, Theorem~\ref{thm:conv} is applicable with $\beta = 0$ in this case.
To prove~\eqref{eqn:x_superlinear}, consider~\eqref{eqn:zt_kt} from the proof of Theorem~\ref{thm:conv}. Upon taking limits on both sides as $t \to \infty$ and substituting $\beta  = 0$, we get
\begin{align}
    \lim_{t \to \infty} \frac{\norm{z(t+1)}}{\norm{z(t)}} \leq \lim_{t \to \infty} \left( \eta \frac{\gamma}{2}\norm{z(t)} + l \norm{\widetilde{K}(t)} \right). \label{eqn:zt_lim}
\end{align}
Since $\mu > 1$, from~\eqref{eqn:x_linear} of Theorem~\ref{thm:conv} we have that $\lim_{t \to \infty} \norm{z(t)} = 0$. Since $\rho < 1$ and $\lim_{t \to \infty} \norm{z(t)} = 0$, from~\eqref{eqn:kt_conv} we have $\lim_{t \to \infty} \norm{\widetilde{K}(t)} = 0$. Upon substituting them in~\eqref{eqn:zt_lim} above, we obtain~\eqref{eqn:x_superlinear}. Hence, the proof.

\end{document}